\documentclass[a4paper,12pt]{amsart}

\author{Giles Gardam}
\title{A counterexample to the unit conjecture for group rings}

\usepackage[T1]{fontenc}
\usepackage[utf8]{inputenc}
\usepackage{amsfonts}
\usepackage{amsmath}
\usepackage{amssymb}
\usepackage{amsthm}
\usepackage{fullpage}
\usepackage{mathpazo}
\usepackage{mathtools}
\usepackage{parskip}
\usepackage{tikz}

\usepackage{hyperref}

\usepackage{setspace}
\theoremstyle{plain}
\newtheorem{thm}{Theorem}
\newtheorem{cor}[thm]{Corollary}

\newtheorem{lem}{Lemma}

\theoremstyle{definition}
\newtheorem*{rmk}{Remark}
\newtheorem{qn}{Question}

\newcommand{\C}{\mathbb{C}}
\newcommand{\F}{\mathbb{F}}
\newcommand{\N}{\mathbb{N}}

\newcommand{\Z}{\mathbb{Z}}
\newcommand{\gp}[2]{\langle \, #1 \, | \, #2 \, \rangle}
\newcommand{\gen}[1]{\langle #1 \rangle}
\newcommand{\abs}[1]{\left\lvert#1\right\rvert}
\newcommand{\half}{\frac{1}{2}}

\DeclareMathOperator{\supp}{supp}
\newcommand{\csupp}[1]{\abs{\supp{ #1 }}}

\newcommand{\overbar}[1]{\mkern 1.5mu\overline{\mkern-1.5mu#1\mkern-1.5mu}\mkern 1.5mu}

\newcommand{\ov}{\overbar{v}}
\newcommand{\ow}{\overbar{w}}

\newcommand{\oz}{\overbar{z}}

\mathtoolsset{centercolon}

\allowdisplaybreaks

\address{Mathematisches Institut, Universität Münster, Einsteinstr.~62, 48149 Münster, Germany}
\email{ggardam@uni-muenster.de}

\dedicatory{To the memory of Willem Henskens}

\begin{document}

\begin{abstract}
    The unit conjecture, commonly attributed to Kaplansky, predicts that if $K$ is a field and $G$ is a torsion-free group then the only units of the group ring $K[G]$ are the trivial units, that is, the non-zero scalar multiples of group elements.
    We give a concrete counterexample to this conjecture; the group is virtually abelian and the field is order two.
\end{abstract}

\maketitle

\section{Introduction}

Three long-standing open problems on the group rings of torsion-free groups are commonly attributed to Kaplansky: the unit conjecture, the zero divisor conjecture and the idempotent conjecture.
Let $K$ be a field and $G$ be a torsion-free group and consider the group ring $K[G]$.
The unit conjecture states that every unit in $K[G]$ is of the form $k g$ for $k \in K \setminus \{ 0 \}$ and $g \in G$, the zero divisor conjecture states that $K[G]$ has no non-trivial zero divisors, and the idempotent conjecture states that $K[G]$ has no idempotents other than $0$ and $1$.
The unit conjecture implies the zero divisor conjecture \cite[Lemma 13.1.2]{Passman85}, which in turn implies the idempotent conjecture; these implications hold for each individual group ring $K[G]$.

In this paper we disprove the unit conjecture.
\begin{thm}
    \label{thm:nontrivial_unit}
    Let $P$ be the torsion-free group defined by the presentation $\gp{a, b}{b^{-1} a^2 b = a^{-2}, \, a^{-1} b^2 a = b^{-2}}$ and set $x = a^2, y = b^2, z = (ab)^2$. Set
    \begin{align*}
        p &= (1 + x) (1 + y) (1 + z^{-1}) \\
        q &= x^{-1} y^{-1}  + x + y^{-1} z + z \\
        r &= 1 + x + y^{-1} z + xyz \\
        s &= 1 + (x + x^{-1} + y + y^{-1}) z^{-1}.
    \end{align*}
    Then $p + qa + rb + sab$ is a non-trivial unit in the group ring $\F_2 [P]$.
\end{thm}

The unit conjecture and the zero divisor conjecture were formulated by Higman in his unpublished 1940 thesis \cite[p.~77]{Higman40b} (see also \cite[p.~112]{Sandling81}), in which he proved the unit conjecture for locally indicable groups (the corresponding paper \cite{Higman40a} was written earlier and omits the conjectures).
The zero divisor conjecture appeared in print in the report of a 1956 talk of Kaplansky \cite[Problem 6]{Kaplansky57}; for integral group rings it appears in the original 1965 Kourovka Notebook as a ``well-known problem'' \cite[1.3]{Kourovka18}.
The unit conjecture was later posed alongside the zero divisor conjecture by Kaplansky \cite{Kaplansky70} having also been asked for integral group rings by Smirnov and Bovdi.

If the unit conjecture were true for $\F_2 [P]$ then its group of units would be isomorphic to $P$ and in particular ``small'' in various senses, such as being finitely generated and virtually abelian; this prediction also fails.
\begin{cor}
    \label{cor:big_group_of_units}
    The group of units of $\F_2 [P]$ is a torsion-free linear group that is not finitely generated and contains non-abelian free subgroups.
\end{cor}

Since this paper first appeared in preprint form, Murray has extended the construction of Theorem~\ref{thm:nontrivial_unit} to give non-trivial units in $\F_d[P]$, for $d$ an arbitrary prime \cite{Murray21}.
This in particular rules out the possibility of giving a model-theoretic compactness proof of the unit conjecture for $\C[P]$ by first establishing it for many finite fields.

In Section~\ref{section:context} we provide further context, explaining our focus on the group $P$.
The proofs are given in Section~\ref{section:counterexample} in 3 parts: deriving how to compute in $K[P]$ in terms of $K[\Z^3]$, giving a criterion for \emph{piecewise symmetric} units, then applying this to prove Theorem~\ref{thm:nontrivial_unit} and Corollary~\ref{cor:big_group_of_units}.
Finally, we briefly discuss open problems in Section~\ref{section:open_problems}.

\begin{rmk}
    After the fact, with the unit and its inverse in hand, Theorem~\ref{thm:nontrivial_unit} is of course readily verified using computer algebra.
    Since it admits a short human-readable proof, we present such a proof.
\end{rmk}

\section{Background}
\label{section:context}

The zero divisor conjecture and the idempotent conjecture have turned out to be susceptible to analytic and $K$-theoretic methods, despite having been posed with very little evidence; for example, the zero divisor conjecture holds for elementary amenable groups \cite[Theorem 1.4]{KrophollerLinnellMoody88} and holds over $\C$ for groups satisfying the (strong) Atiyah conjecture (see \cite{Linnell93}, \cite[Lemma 10.39]{Lueck02}), and the idempotent conjecture over $\C$ follows from either the Baum--Connes or Farrell--Jones conjecture (see \cite[\S 6.3]{Valette02} and \cite[Theorem 1.12]{BartelsLueckReich08} respectively).
These stronger conjectures are in turn known to hold for different and quite large classes of groups; for instance, the Baum--Connes conjecture is known for hyperbolic \cite{MineyevYu02} \cite{Lafforgue02} and amenable (even a-T-menable) groups \cite{HigsonKasparov01} and the Farrell--Jones conjecture is known for hyperbolic and CAT(0) groups \cite{BartelsLueck12} as well as lattices in virtually connected Lie groups \cite{KammeyerLueckRueping16}.
In contrast, the unit conjecture has only been established as a consequence of the stronger combinatorial and purely group-theoretic property of having unique products.
A group $G$ is said to \emph{have unique products} if for every choice of non-empty finite subsets $A, B \subset G$ the set $A \cdot B = \{ a b : a \in A, b \in B \}$ contains some element uniquely expressible as $a b$ for $a \in A, b \in B$.
This implies the \emph{a~priori} stronger ``two unique products'' property \cite{Strojnowski80} and thus the unit conjecture.

The first example of a torsion-free group without unique products was constructed by Rips and Segev using small cancellation techniques \cite{RipsSegev87}.
Shortly thereafter, Promislow provided an elementary example \cite{Promislow88} in the (torsion-free) virtually abelian group $P$ variously known as the Hantzsche--Wendt group, Promislow group, or Fibonacci group $F(2, 6)$.
The group $P$ is the unique torsion-free 3-dimensional crystallographic group with finite abelianization; specifically, it is a non-split extension \[
    1 \to \Z^3 \to P \to \Z/2 \times \Z/2 \to 1.
\]
In Promislow's example we have $A = B$ and $\abs{A} = 14$.

The group $P$ satisfies the Farrell--Jones conjecture (indeed, it is the fundamental group of the flat manifold of Hantzsche--Wendt so it is covered by the original work of Farrell--Jones \cite{FarrellJones93} for integral coefficients and by \cite{BartelsLueck12} for general coefficients)  so one interpretation of Theorem~\ref{thm:nontrivial_unit} is that the unit conjecture really is immune to algebraic $K$-theory after all.
However, it is important to note that the Farrell--Jones conjecture for a group $G$ implies that a certain stable version of the unit conjecture holds, whereby every unit can be trivialized in $K_1(K[G]) = \operatorname{GL}(K[G])^{\operatorname{ab}}$, as the corresponding Whitehead group $\operatorname{Wh}^{K}(G)$ is trivial (see \cite[Theorem 1.2(iii)]{BartelsLueckReich08} and the discussion preceding it).

Craven and Pappas attacked the question of whether group rings of $P$ have non-trivial units, filtering potential units according to a complexity measure called length $L$ that comes from the word length in the infinite dihedral quotient of $P$.
They established the unit conjecture for $P$ amongst elements of length $L \leq 3$ which as an application shows that Promislow's set does not support a non-trivial unit \cite[Theorem 12.1]{CravenPappas13}.
We have to entertain the possibility that the unit conjecture is equivalent to unique products; if it is, and failure of unique products begets non-trivial units, then the aforementioned theorem shows that this cannot happen ``locally'' in the corresponding sense.
Our counterexample has length $L = 4$ (after conjugation by $a$).

It appears plausible that non-trivial units in $\F_2 [P]$ always have support of cardinality at least $21$, as in Theorem~\ref{thm:nontrivial_unit}.
This is prohibitively large for previous strategies to find such a unit, which are subject to some form of combinatorial explosion or another.
One approach is to consider how products could cancel in pairs to give a unit over $\F_2$ and take the presentation defining a corresponding universal group; this was used for instance to rule out the existence of non-trivial units $\alpha$ over $\F_2$ with $\csupp{\alpha} = 3$ and $\csupp{\alpha^{-1}} \leq 11$ or $\csupp{\alpha} = \csupp{\alpha^{-1}} = 5$ for any torsion-free group \cite[Proposition 4.12]{DykemaHeisterJuschenko15}.
Alternatively, one could work directly with a specific group, such as $P$.
This has surely been attempted before, and one can decide invertibility of $\alpha \in K[P]$ using a determinant condition \cite[Theorem 8.5]{CravenPappas13}, but the number of candidates is astronomical.
A key factor that makes the counterexample tick is that both it and its inverse are assembled from the same highly symmetric pieces (see Lemma~\ref{lem:symmetry}).
Note however that although Promislow's undocumented ``random search algorithm'' could have been restricted to $A = B$ without losing completeness \cite{Strojnowski80}, greatly constraining the search space, a unit $\alpha \neq 1$ in $\F_2 [P]$ cannot be self-inverse since then $(\alpha + 1)^2 = 0$.
The computational method used to find this unit via Boolean satisfiability will be discussed in \cite{Gardam21+}.

The zero divisor conjecture is known for $P$ so this counterexample does not directly suggest a line of attack.
If the zero divisor conjecture is true, then we have established that it is not true for the perhaps more combinatorial considerations of the unit conjecture.
On the other hand, if the zero divisor conjecture is false, then this paper at least removes one psychological impediment to finding a counterexample.

\section{The counterexample}
\label{section:counterexample}

\subsection{Setup}

In this paper we will work with the structure of $P$ as an extension of $\Z^3$ by a $\Z/2 \times \Z/2$ quotient.
In order to facilitate calculations, we will describe this extension very explicitly, including its defining cocycle (factor set) and action.
One could perform this computation in various ways; our approach is flavoured with Bass--Serre theory.

We adopt the convention of conjugation acting on the right: $s^t := t^{-1} s t$.
We introduce new variables $x = a^2$ and $y = b^2$ into the presentation \[
    \gp{x, b, y, a}{ x^b = x^{-1}, \, y^a = y^{-1}, \, x = a^2, \, b^2 = y }
\] and observe that this expresses $P$ as an amalgam of two Klein bottle groups, namely $\gp{x, b}{ x^b = x^{-1} }$ and $\gp{y, a}{y^a = y^{-1}}$, along their isomorphic index-2 $\Z^2$ subgroups $\gen{x, b^2} = \gen{a^2, y}$ (in particular, this shows that $P$ is torsion-free).
Being normal in each factor, this subgroup is normal in the amalgam, with corresponding quotient $D_\infty = \Z/2 * \Z/2$.
We pick $z = abab$ as a lift to $P$ of a generator of the infinite cyclic group $[D_\infty, D_\infty]$.
As $x^z = (a^2)^{(abab)} = (a^{-2})^{(ab)} = x$ and similarly $y^z = y$, we see that in fact $\gen{x, y, z} \cong \Z^3$ is the kernel of $P \to \Z/2 \times \Z/2$.

Write $Q$ for the quotient $\Z/2 \times \Z/2$.
The action of $P$ by conjugation on $\langle x, y, z \rangle$ induces an action of $Q$ in which the non-trivial elements act as conjugation by $a$, $b$ and $ab$.
The action of $a$ and $b$ on $\gen{x, y}$ can be read off the presentation; for the action on $z$ note that $z^{ab} = z$ and \[
    z^a z = b a b (a^2 b) a b = b a b (b a^{-2}) a b = b (a b^2) a^{-1} b = b (b^{-2} a) a^{-1} b = 1.
\]
Let us explicitly record the action for all 3 non-trivial elements of $Q$.
\begin{align*}
    x^a    &= x     &    y^a &= y^{-1} &    z^a &= z^{-1} \\
      x^b  &= x^{-1}&    y^b &= y      &    z^b &= z^{-1} \\
    x^{ab} &= x^{-1}& y^{ab} &= y^{-1} & z^{ab} &= z
\end{align*}
The set-theoretic section $\sigma \colon Q \to P$ with image $\{1, a, b, ab\}$ defines the cocycle $f \colon Q \times Q \to \Z^3$ by $f(g, h) = \sigma(g) \sigma(h) \sigma(gh)^{-1}$.
In order to compute it we just need to know how to push an $a$ past a $b$.
One of the defining relations tells us that $b^{-1} a^2 = a^{-2} b^{-1}$ and thus \[
    b a b^{-1} a^{-1}
    = b^2 (b^{-1} a^2) a^{-1} b^{-1} a^{-1}
    = b^2 (a^{-2} b^{-1}) a^{-1} b^{-1} a^{-1}
    = y x^{-1} z^{-1}
    = x^{-1} y z^{-1}.
\]
With this identity in hand we determine
\begin{center}
    \begin{tabular}{r | c c c c}
        $f(g, h)$ & $1$ & $a$ & $b$ & $ab$ \\
    \hline
        $1$  & $1$ & $1$ & $1$ & $1$ \\
        $a$  & $1$ & $x$ & $1$ & $x$ \\
        $b$  & $1$ & $x^{-1} y z^{-1}$ & $y$ & $x^{-1} z^{-1}$ \\
        $ab$ & $1$ & $y^{-1} z$ & $y^{-1}$ & $z$ \\
    \end{tabular}
\end{center}
where the table reads left-to-right (the rows give $g$ and the columns give $h$).

Let $K$ be a field.
Every element $\beta \in K[P]$ can be written uniquely as $(\beta)_1 + (\beta)_a a + (\beta)_b b + (\beta)_{ab} ab$ for 4 Laurent polynomials in $K[x^{\pm 1}, y^{\pm 1}, z^{\pm 1}]$.
Let us now record the general product in $K[P]$ in terms of such polynomials, which follows immediately from the above table giving the cocycle $f$.
If $\alpha = p + qa + rb + sab$ and $\alpha' = p' + q'a + r'b + s'ab$ for Laurent polynomials $p, q, r, s, p', q', r', s' \in K[x^{\pm 1}, y^{\pm 1}, z^{\pm 1}]$, then
    \begin{align*}
        (\alpha' \alpha)_1    &= p' p + x q' q^a + y r' r^b + z s' s^{ab} \\
        (\alpha' \alpha)_a    &= p' q + q' p^a   + x^{-1} z^{-1} r' s^b + y^{-1} s' r^{ab} \\
        (\alpha' \alpha)_b    &= p' r + x q' s^a + r' p^b + y^{-1} z s' q^{ab} \\
        (\alpha' \alpha)_{ab} &= p' s + q' r^a   + x^{-1} y z^{-1} r' q^b + s' p^{ab}.
    \end{align*}

\subsection{Piecewise symmetric units}

In order to express fully the symmetry exhibited by the constituent polynomials of the non-trivial unit, we need to introduce square roots into the polynomial ring.
Morally, one could see this as a form of fake torsion, since at the group level this would correspond to taking a finite index overgroup of $P$ which is no longer torsion-free.
(Note that we do not mean symmetry in the sense of symmetric polynomials, which is \emph{permutational}, but rather specific \emph{rotational} symmetry, as per the statement of the lemma.)

\begin{lem}
    \label{lem:symmetry}
    Let $K$ be a field and let the group $P$ and $a, b, x, y, z \in P$ be as in Theorem~\ref{thm:nontrivial_unit}.
    Let $p, q, r, s \in K[x^{\pm 1}, y^{\pm 1}, z^{\pm 1}]$ be Laurent polynomials in $x, y, z$.
    Adjoin square roots of $x$ and $y$ and suppose that $p_0, q_0, r_0, s_0 \in K[x^{\pm \half}, y^{\pm \half}, z^{\pm 1}]$ defined by $p_0 = x^{-\half} y^{-\half} p$, $q_0 = y^\half q$, $r_0 = x^{-\half} r$, $s_0 = s$ are all invariant under the action of $ab$ (i.e. under $x \mapsto x^{-1}, y \mapsto y^{-1}, z \mapsto z$).
    If the equations
    \begin{align}
        p_0^a s_0 - q_0 r_0^a + z^{-1} (p_0^a s_0 - q_0 r_0^a)^a &= 0 \label{align:ab} \\
        p_0 p_0^a - q_0 q_0^a - r_0 r_0^a + s_0 s_0^a &= 1 \label{align:1}
    \end{align}
    hold, then $p + qa + rb + sab$ is a unit in $K[P]$.
\end{lem}

\begin{proof}
    Let $\alpha = p + qa + rb + sab$ and $\alpha' = p' + q'a + r'b + s'ab$.
    We shall prove that the following choices $p', q', r', s'$ make $\alpha'$ inverse to $\alpha$.
    Since $P$ is virtually abelian it satisfies the zero divisor conjecture \cite[Theorem 2]{Cliff80}, so it suffices to check that $\alpha' \alpha = 1$ (as this implies that $\alpha' (\alpha \alpha' - 1) = 0$ and thus $\alpha \alpha' = 1$).
    In this table we also record the action of $ab$ and $b$ on the polynomials $p, q, r, s$, which is an immediate consequence of the invariance of $p_0, q_0, r_0$ and $s_0$ under the action of $ab$.
    \begin{align*}
        p' &:= x^{-1} p^a    &  p^{ab} &= x^{-1} y^{-1} p  & p^b &= x^{-1} y p^a \\
        q' &:= -x^{-1} q     &  q^{ab} &= y q              & q^b &= y^{-1} q^a   \\
        r' &:= -y^{-1} r     &  r^{ab} &= x^{-1} r         & r^b &= x^{-1} r^a   \\
        s' &:= z^{-1} s^a    &  s^{ab} &= s                & s^b &= s^a
    \end{align*}
    The invariance encoded in these last two columns does half of the computation of $\alpha' \alpha$: substituting in immediately gives
    \begin{align*}
        (\alpha' \alpha)_a &= x^{-1} p^a q - x^{-1} q p^a - x^{-1} z^{-1} y^{-1} r s^a + y^{-1} z^{-1} s^a x^{-1} r = 0, \\
        (\alpha' \alpha)_b &= x^{-1} p^a r - x x^{-1} q s^a - y^{-1} r x^{-1} y p^a + y^{-1} z z^{-1} s^a y q = 0.
    \end{align*}
    For the remaining two components, we have
    \begin{align*}
        (\alpha' \alpha)_{ab}
        &= x^{-1} p^a s - x^{-1} q r^a - x^{-1} y z^{-1} y^{-1} r y^{-1} q^a + z^{-1} s^a x^{-1} y^{-1} p \\
        &= x^{-1} (p^a s - q r^a - y^{-1} z^{-1} r q^a + y^{-1} z^{-1} s^a p) \\
        &= x^{-1} (x^\half y^{-\half} p_0^a s_0 - y^{-\half} q_0 x^\half r_0^a - y^{-1} z^{-1} x^\half r_0 y^{\half} q_0^a + y^{-1} z^{-1} s_0^a x^\half y^\half p_0) \\
        &= x^{-\half} y^{-\half} (p_0^a s_0 - q_0 r_0^a - z^{-1} r_0 q_0^a + z^{-1} s_0^a p_0) \\
        &= x^{-\half} y^{-\half} (p_0^a s_0 - q_0 r_0^a + z^{-1} (p_0^a s_0 - q_0 r_0^a)^a) \\
        &= 0 \\
    \noalign{\noindent and}
        (\alpha' \alpha)_{1}
        &= x^{-1} p^a p - x x^{-1} q q^a - y y^{-1} r x^{-1} r^a + z z^{-1} s^a s \\
        &= x^{-1} x^\half y^{-\half} p_0^a x^\half y^\half p_0 - y^{-\half} q_0 y^\half q_0^a - x^{-1} x^\half r_0 x^\half r_0^a + s_0^a s_0 \\
        &= p_0 p_0^a - q_0 q_0^a - r_0 r_0^a + s_0 s_0^a \\
        &= 1.
    \end{align*}
\end{proof}

\subsection{Proofs of Theorem~\ref{thm:nontrivial_unit} and Corollary~\ref{cor:big_group_of_units}}

\begin{proof}[Proof of Theorem~\ref{thm:nontrivial_unit}]
    For notational convenience, let $v = x^\half$ and $w = y^\half$, and write $\ov = v^{-1} + v$, $\ow = w^{-1} + w$ and $\oz = z^{-1} + z$.
    We simply need to verify that the polynomials
    \begin{align*}
        p_0 &= \ov \ow (1 + z^{-1}) \\
        q_0 &= v^{-2} w^{-1} + v^2 w + \ow z \\
        r_0 &= \ov + (v^{-1} w^{-2} + v w^2) z \\
        s_0 &= 1 + (v^{-2} + v^2 + w^{-2} + w^2) z^{-1} \\
            &= 1 + (\ov^2 + \ow^2) z^{-1}
    \end{align*}
    satisfy Equations~\eqref{align:ab} and \eqref{align:1} of Lemma~\ref{lem:symmetry}, since they evidently have the required symmetry under simultaneous inversion of $v$ and $w$ (that is, the action by $ab$).
    Note that since we are working in characteristic $2$, $\ov^2 = (v^{-1} + v)^2 = v^{-2} + v^2$.

    Let $\xi := p_0^a s_0 - q_0 r_0^a = p_0^a s_0 + q_0 r_0^a$ in characteristic 2.
    Define the map $\psi \colon \F_2 [v^{\pm 1}, w^{\pm 1}, z^{\pm 1}] \to \F_2 [v^{\pm 1}, w^{\pm 1}, z^{\pm 1}]$ by $\psi(\eta) = z^{-1} \eta^a$.
    We show that $\psi(\xi) = \xi$, so that Equation~\eqref{align:ab} holds as claimed, by modifying each of the summands of $\xi$ by $\epsilon := \ov \ow (1 + z)$.
    Note first that
    \begin{align*}
        p_0^a s_0 + \epsilon
        &= \ov \ow (1 + z) (1 + (\ov^2 + \ow^2) z^{-1}) + \ov \ow (1 + z) \\
        &= \ov \ow (1 + z^{-1}) (\ov^2 + \ow^2)
    \end{align*}
    is invariant under $\psi$.
    Write $\gamma = v^{-2} w^{-1} + v^2 w$ and $\delta = v^{-1} w^{-2} + v w^2$.
    The invariance of
    \begin{align*}
        q_0 r_0^a + \epsilon
        &= (\gamma + \ow z)(\ov + \delta^a z^{-1}) + \ov \ow (1 + z) \\
        &= (\ov \gamma + \ow \delta^a + \ov \ow)z^0 + \gamma \delta^a z^{-1}
    \end{align*}
    under $\psi$ now follows from
    \begin{align*}
        \ov \gamma + \ow \delta^a + \ov \ow
        &= (v^{-1} + v)(v^{-2} w^{-1} + v^2 w) + (w^{-1} + w) (v^{-1} w^2 + v w^{-2}) + \ov \ow \\
        &= v^{-3} w^{-1} + v^3 w + v w^{-3} + v^{-1} w^3 + 2 \ov \ow \\
        &= (v^{-2} w + v^2 w^{-1}) (v^{-1} w^{-2} + v w^2) \\
        &= \gamma^a \delta.
    \end{align*}
    Thus $\psi(\xi) = \psi(p_0^a s_0 + \epsilon + q_0 r_0^a + \epsilon) = \xi$, establishing Equation~\eqref{align:ab}.

    We observe that \[
        \gamma \gamma^a = (v^{-2} w^{-1} + v^2 w) (v^{-2} w + v^2 w^{-1}) = \ov^4 + \ow^2
    \] and similarly $\delta \delta^a = \ow^4 + \ov^2$. Hence we can compute
    \begin{align*}
        p_0 p_0^a
            &= \ov \ow (1 + z^{-1}) \ov \ow ( 1 + z) \\
            &= \ov^2 \ow^2 \oz \\
        q_0 q_0^a
        &= (\gamma + \ow z) (\gamma^a + \ow z^{-1}) \\
            &= (\ov^4 + \ow^2) + \ow (\gamma z^{-1} + \gamma^a z) + \ow^2 \\
            &= \ov^4 + \ow (\gamma z^{-1} + \gamma^a z) \\
        r_0 r_0^a
        &= (\ov + \delta z) (\ov + \delta^a z^{-1}) \\
        &= \ov^2 + \ov (\delta^a z^{-1} + \delta z) + (\ow^4 + \ov^2) \\
        &= \ow^4 + \ov (\delta^a z^{-1} + \delta z) \\
        s_0 s_0^a
        &= (1 + (\ov^2 + \ow^2) z^{-1}) (1 + (\ov^2 + \ow^2) z) \\
        &= 1 + (\ov^2 + \ow^2)^2 + (\ov^2 + \ow^2) \oz \\
        &= 1 + \ov^4 + \ow^4 + (\ov^2 + \ow^2) \oz.
    \end{align*}

    Since
    \begin{align*}
        \ow \gamma + \ov \delta^a
        &= (w^{-1} + w) (v^{-2} w^{-1} + v^2 w) + (v^{-1} + v)(v^{-1} w^2 + v w^{-2}) \\
        &= v^{-2} w^{-2} + v^2 + v^{-2} + v^2 w^2 + v^{-2} w^2 + w^{-2} + w^2 + v^2 w^{-2} \\
        &= \ov^2 \ow^2 + \ov^2 + \ow^2 \\
        &= \ow \gamma^a + \ov \delta
    \end{align*}
    we see that $q_0 q_0^a + r_0 r_0^a = \ov^4 + \ow^4 + (\ov^2 \ow^2 + \ov^2 + \ow^2) \oz$.
    Thus
    \begin{align*}
        & p_0 p_0^a + (q_0 q_0^a + r_0 r_0^a) + s_0 s_0^a \\
        &= \ov^2 \ow^2 \oz + (\ov^4 + \ow^4 + (\ov^2 \ow^2 + \ov^2 + \ow^2) \oz) + (1 + \ov^4 + \ow^4 + (\ov^2 + \ow^2) \oz) \\
        &= 1
    \end{align*}
    completing the verification of Equation~\eqref{align:1} and thus the proof.
\end{proof}

\begin{proof}[Proof of Corollary~\ref{cor:big_group_of_units}]
    We first observe that the group $E = (\F_2 [P])^\times$ of units is torsion-free and linear (which is unrelated to our counterexample and presumably already known).
    Torsion-freeness follows quickly from the absence of zero divisors.
    Indeed, if $\beta \neq 1$ and $\beta^n = 1$ then factorizing $\beta^n - 1$ implies that $\beta^{n-1} + \dots + 1$ is zero.
    If $n = 2$ this is just $\beta = 1$ and if $n$ is odd then this gives a contradiction after applying the augmentation map $\F_2 [P] \to \F_2$.
    It is a classical fact that if $H$ is an index $n$ subgroup of $G$ then $K[G] \subseteq M_n(K[H])$ \cite[Lemma 5.1.10]{Passman85} which in our case restricts to $E \leq \operatorname{GL}_4 (\F_2 [x^{\pm 1}, y^{\pm 1}, z^{\pm 1}])$ so $E$ is linear over the field $\F_2(x, y, z)$.

    We will now consider an infinite dihedral quotient of $P$ to verify that $E$ is infinitely generated and has free subgroups.
    Let $D_\infty = \gp{t, b}{b^2 = 1, \, t^b = t^{-1}}$.
    The group of units of $\F_2 [D_\infty]$ was determined by Mirowicz \cite[Theorem 4.1]{Mirowicz91} and we now recall its description.
    For $i \in \N^+$ and $j \in \Z$ let \[
        e_{ij} = t^{-i} + 1 + t^{i} + t^j(t^{-i} + t^i)b.
    \]
    Then $\{ e_{ij} : i \in \N^+ \}$ is a basis for an elementary abelian 2-group $U_j \cong \oplus_{i \in \N^+} \Z/2$ and the subgroup generated by all $e_{ij}$ is $*_{j \in \Z} U_j$.
    The trivial units $D_\infty$ normalize this free product, acting by $e_{i,j}^b = e_{i,-j}$ and $e_{i,j}^t = e_{i,j-2}$.
    The full group of units is in fact the semidirect product \[
        (\F_2 [D_\infty])^\times = (*_{j \in \Z} U_j) \rtimes D_\infty.
    \]
    Note that mapping $e_{i,j} \mapsto e_{i, 0}, t \mapsto 1, b \mapsto 1$ gives a well-defined retraction $(\F_2 [D_\infty])^\times~\to~U_0$; since $U_0$ is elementary abelian, this implies that $(\F_2 [D_\infty])^\times$ retracts onto any subgroup of $U_0$.

    For any unit given by Lemma~\ref{lem:symmetry}, we can replace $x^\half$ with $x^{k+\half}$ in the 4 polynomials $p_0, q_0, r_0, s_0$ and still have a solution to the required equations (in which the occurrence of $z$ is the only non-homogeneity) and in the case of Theorem~\ref{thm:nontrivial_unit} this gives a family of units
    \begin{align*}
        \alpha_k
            &:= (x^{-k} + x^{k+1})(1 + y)(1 + z^{-1}) + \left( x^{-2k-1}y^{-1} + x^{2k+1} + (y^{-1} + 1) z \right)a \\
            &+ \left( x^{-k} + x^{k+1} + (x^{-k} y^{-1} + x^{k+1} y)z \right) b + \left( 1 + (x^{-2k-1} + x^{2k+1} + y^{-1} + y)z^{-1} \right) ab
    \end{align*}
    modelled on $\alpha = \alpha_0$.
    The surjection $\pi \colon P \to D_\infty \colon a \mapsto t, b \mapsto b$, with $\pi(x) = t^2$ and $\pi(y) = \pi(z) = 1$, induces a ring homomorphism $\pi \colon \F_2 [P] \to \F_2 [D_\infty]$ and the unit $\alpha_k$ is mapped to \[
        \pi (\alpha_k) = 0 + (t^{-4k-2} + t^{4k+2}) t + (t^{-2k} + t^{2k+2})(1 + 1)b + (1 + t^{-4k-2} + t^{4k+2}) t b
    \] and thus $\pi(a^{-1} \alpha_k b) = e_{4k+2,0}$.
    Thus the image $\pi(E) \subseteq (\F_2 [D_\infty])^\times$ contains a subgroup isomorphic to $\oplus_{\N^+} \Z/2$ onto which $(\F_2 [D_\infty])^\times$ retracts, so $E$ is not finitely generated.
    Indeed, we have shown that $H^1(E, \F_2)$ is infinite dimensional.

    Finally, we can for instance take the conjugate units $\alpha b a^{-1}$, $a^{-1} \alpha b$ and $a^{-2} \alpha b a$, whose images in $(\F_2 [D_\infty])^\times$ are $e_{2,2}$, $e_{2,0}$ and $e_{2,-2}$, generating $\Z/2 * \Z/2 * \Z/2$, so $E$ contains non-abelian free subgroups.
\end{proof}

\section{Discussion}
\label{section:open_problems}

This paper closes the question of whether the unit conjecture holds for all torsion-free groups but opens up the study of the corresponding groups of units; Corollary~\ref{cor:big_group_of_units} is a first step in this direction, but determining all units in $\F_2 [P]$ seems extremely difficult, especially since the group of units is not finitely generated and there is no apparent reason why all units should be attainable from this counterexample and its obvious variants.
Finding non-trivial units for the group $P$ in characteristic zero remains an outstanding challenge.

The group $P$ is in many senses the simplest group where the unit conjecture could fail.
Although the construction in this paper is highly specialized, it should nonetheless be possible to find non-trivial units for torsion-free groups that do not contain $P$ as a subgroup.
This brings the unique product property into focus.
Torsion-free groups without unique products constructed via small cancellation arguments, such as in \cite{RipsSegev87,Steenbock15,GruberMartinSteenbock15,ArzhantsevaSteenbock14}, pose a difficulty for hands-on computational approaches, since the presentations are correspondingly large, and we know very few examples beyond these \cite{Carter14,Soelberg18}.

A better, more geometric understanding of the unique product property would be illuminating.
The property of being \emph{diffuse}, introduced by Bowditch \cite{Bowditch00}, is a consequence of left-orderability and implies the unique product property.
Dunfield gave the first example of a diffuse group which is not left-orderable \cite[Theorem A.1]{KionkeRaimbault16}.
\begin{qn}[{\cite[Question 1]{KionkeRaimbault16}}]
    Does there exist a group which is not diffuse but has unique products?
\end{qn}

Distinguishing the unique product property from the unit conjecture is likewise outstanding.
\begin{qn}
    Does there exist a torsion-free group $G$ without the unique product property such that $K[G]$ satisfies the unit conjecture, for some field $K$?
\end{qn}

\subsection*{Acknowledgements}

This work was supported by the Deutsche Forschungsgemeinschaft (DFG, German Research Foundation) -- Project-ID 427320536 -- SFB 1442, as well as under Germany's Excellence Strategy EXC 2044--390685587, Mathematics M\"unster: Dynamics--Geometry--Structure and within the Priority Programme SPP2026 ``Geometry at Infinity''.

I thank Arthur Bartels for helpful conversations and Martin Bridson for comments on a draft.
I am grateful to Kenneth A. Brown and Sven Raum for bringing to my attention that Higman formulated the conjecture.
I became aware of Mirowicz's paper via Yves Cornulier and Salvatore Siciliano on MathOverflow.
I thank the anonymous referee who made insightful comments on the context of this paper.

\bibliography{units}{}
\bibliographystyle{alpha}

\end{document}